\documentclass[12pt]{article}%
\usepackage{amsfonts}
\usepackage{mitpress}
\usepackage{amsmath}
\usepackage{amssymb}
\usepackage{graphicx}%
\providecommand{\U}[1]{\protect\rule{.1in}{.1in}}
\newtheorem{theorem}{Theorem}

\newtheorem{proposition}[theorem]{Proposition}
\newtheorem{remark}[theorem]{Remark}

\newenvironment{proof}[1][Proof]{\noindent\textbf{#1.} }{\ \rule{0.5em}{0.5em}}
\newdimen\dummy
\dummy=\oddsidemargin
\begin{document}

\title{Upper bounds for the error in some interpolation and extrapolation designs}
\author{Michel Broniatowski$^{1}$,Giorgio Celant$^{2}$, Marco Di Battista$^{2}$ ,
Samuela Leoni-Aubin$^{3}$}
\maketitle
\title{Upper bounds for the error in some interpolation and extrapolation designs}
\author{Michel Broniatowski$^{1}$,Giorgio Celant$^{2}$, Marco Di Battista$^{2}$ ,
Samuela Leoni-Aubin$^{3}$\\$^{1}${\footnotesize Universit\'{e} Pierre et Marie Curie,
LSTA,e-mail:michel.broniatowski@upmc.fr}\\$^{2}${\footnotesize University of Padua, Department of Statistical Sciences,}\\{\footnotesize \noindent}$^{3}${\footnotesize INSA Lyon, ICJ, e-mail:
samuela.leoni@insa-lyon.fr}}
\maketitle

\begin{abstract}
This paper deals with probabilistic upper bounds for the error in functional
estimation defined on some interpolation and extrapolation designs, when the
function to estimate is supposed to be analytic. The error pertaining to the
estimate may depend on various factors: the frequency of observations on the
knots, the position and number of the knots, and also on the error committed
when approximating the function through its Taylor expansion. When the number
of observations is fixed, then all these parameters are determined by the
choice of the design and by the choice estimator of the unknown function.

AMS (2010) Classification: 62K05, 41A5

\end{abstract}

\noindent

\section{Introduction\bigskip\ }

Consider a function $\varphi$ defined on some open set $D \subset\mathbb{R}$
and which can be observed on a compact subset $S$ included in $D$. The problem
that we consider is the estimation of this function through some interpolation
or extrapolation techniques. This turns out to define a finite set of points
$s_{i}$ in a domain $\tilde{S}$ included in $S$ and the number of measurement
of the function $\varphi$ at each of these points, that is to define a design
$\mathcal{P}:=\left\{  \left(  s_{i},n_{i}\right)  \in S\times%
\mathbb{N}
, ~i=0,...,l,~\widetilde{S}\subsetneqq S\right\}  $. The points $s_{i}$ are
called the \emph{knots}, $n_{i}$ is the frequency of observations at knot
$s_{i}$ and $l+1$ is the number of knots. \noindent The choice of the design
$\mathcal{P}$ is based on some optimality criterion. For example, we could
choose an observation scheme that minimize the variance of the estimator of
$\varphi$.

\noindent The choice of $\mathcal{P}$ has been investigated by many authors.
Hoel and Levine and Hoel (\cite{10} and \cite{11}) considered the case of the
extrapolation of a polynomial function with known degree in one and two
variables. Spruill, in a number of papers (see \cite{17}, \cite{18}, \cite{19}
and \cite{20}) proposed a technique for the (interpolation and extrapolation)
estimation of a function and its derivatives, when the function is supposed to
belong to a Sobolev space, Celant (in \cite{4} and \cite{5}) considered the
extrapolation of quasi-analytic functions and Broniatowski-Celant in \cite{3}
studied optimal designs for analytic functions through some control of the bias.

\noindent The main defect of any interpolation and extrapolation scheme is its
extreme sensitivity to the uncertainties pertaining to the values of $\varphi$
on the knots. The largest the number $l+1$ of knots, the more unstable is the
estimate. In fact, even when the function $\varphi$ is accurately estimated on
the knots, the estimates of $\varphi$ or of one of its derivatives
$\varphi^{(j)}$ at some point in $D$ may be quite unsatisfactory, due either
to a wrong choice of the number of knots or to their location. The only case
when the error committed while estimating the values $\varphi(s_{i})$ is not
amplified in the interpolation procedure is the linear case. Therefore, for
any more envolved case the choice of $l$ and $(s_{i}, n_{i} )$ must be handled
carefully, which explains the wide literature devoted to this subject. For
example, if we estimate $\varphi\left(  v\right)  ,v\in S\diagdown
\widetilde{S},$ by $\widehat{\varphi\left(  s_{k}\right)  }:=\varphi\left(
s_{k}\right)  +\varepsilon\left(  k\right)  ,$ where $\varepsilon\left(
k\right)  $ denotes the estimation error and $\widetilde{S}$ a Tchebycheff set
of points $S $, we obtain
\[
\left\vert \varphi\left(  v\right)  -\widehat{\varphi\left(  s_{k}\right)
}\right\vert \leq\left(  \max_{k}\left\vert \varepsilon\left(  k\right)
\right\vert \right)  \Lambda_{l}\left(  v,s_{k},0\right)  ,
\]
where $\Lambda_{l}\left(  v,s_{i},j\right)  $ is a function that depends on
$\widetilde{S}$, the number of knots and on the order of the derivative that
we aim to estimate (here $0$), and (see \cite{2} and \cite{14} )
\[
\max_{k=0,...,l}\Lambda_{l}\left(  v,s_{k},0\right)  :=\frac{1}{l+1}\sum
_{k=0}^{l}ctg\left(  \frac{2k+1}{4\left(  l+1\right)  }\pi\right)  \sim
\frac{2}{\pi}\ln\left(  l+1\right)  \quad\text{ when }l\rightarrow\infty.
\]
If equidistant knots are used, one gets (see \cite{16})
\[
\max_{k=0,...,l}\Lambda_{l}\left(  v,s_{k},0\right)  \thicksim\frac{2^{l+1}%
}{el\left(  \ln l+\gamma\right)  }, \qquad\gamma
=0,577~~\text{(Euler-Mascheroni constant).}%
\]

\noindent When the bias in the interpolation is zero, as in the case when
$\varphi$ is polynomial with known degree, the design is optimized with
respect to the variance of the interpolated value (see \cite{10}
). In the other cases the criterion that is employed is the minimal MSE
criterion. The minimal MSE criterion allows the estimator to be as accurate as
possible but it does not yield any information on the
interpolation/extrapolation error.

In this paper, we propose a probabilistic tool (based on the concentration of
measure) in order to control the estimation error. In Section 2 we present the
model, the design and the estimators. Section 3 deals with upper bouns for the
error. Concluding remarks are given in Section 4.

\section{The model, the design and the estimators}

Consider an unknown real-valued analytic function $f$ defined on some interval
$D$ :
\begin{align*}
f:~~ D:=\left(  a,b\right)   &  \rightarrow\mathbb{R}\\
v  &  \mapsto f\left(  v\right)  .
\end{align*}

\noindent We assume that this function is observable on a compact subset $S$
included in $D$, $S:=\left[  \underline{s},\overline{s}\right]  \subset D$,
and that its derivatives are not observable at any point of $D$. Let
$\widetilde{S}:=\left\{  s_{k\text{ }}\in\widetilde{S},k=0,...,l\right\}  $ be
a finite subset of $l+1$ elements in the set $S$. The points $s_{k}$ are
called the \emph{knots}.

\noindent Observations $Y_{i}$, $i=1, \ldots,n$ are generated from the
following location-scale model
\begin{align*}
Y_{j}\left(  s_{k}\right)  =  &  f\left(  s_{k}\right)  +\sigma E\left(
Z_{j}\right)  +\varepsilon_{j} ,\\
\varepsilon_{j} :=  &  \sigma Z_{j}-\sigma E\left(  Z_{j}\right)  , \qquad
j=1,...,n_{k},~~k=0,...,l,
\end{align*}
\noindent where $Z$ \ is a completely specified continuous random variable,
the location parameter $f\left(  v\right)  $ and the scale parameter $\sigma>0
$ are unknown parameters. $E\left(  Z\right)  ,\varsigma$ respectively denote
the mean and the variance of $Z$, and $n_{k}$ is the \emph{frequency of
observations} at knot $s_{k}$.

\noindent We assume to observe $(l+1)$ i.i.d. samples, $\underline{Y}\left(
k\right)  :=\left(  Y_{1}\left(  n_{k}\right)  ,...,Y_{n_{k}}\left(
n_{k}\right)  \right)  ,k=0,...,l, $ and $Y_{i}\left(  n_{k}\right)  $ i.i.d.
$Y_{1}\left(  n_{k}\right)  ,$ for all $i\neq k$, $i= 0, \ldots, l$.

\noindent The aim is to estimate a derivative of $f\left(  v\right)  $,
$f^{(d)}(v)$, $d \in\mathbb{N}$, at a point $v \in\left(  a,\overline
{s}\right)  $.

\noindent Let $\varphi\left(  v\right)  :=f\left(  v\right)  +\sigma E\left(
Z\right)  $, and consider the Lagrange polynomial%

\[
L_{s_{k}}\left(  v\right)  :=\prod_{j\neq k,j=0}^{l}\frac{v-s_{j}}{s_{k}%
-s_{j}}.
\]

\noindent We are interested in interpolating (or extrapolating) some
derivatives of $\varphi$, $\varphi^{(d)}$, with $d \in\mathbb{N}$,
\[
\mathcal{L}_{l}\left(  \varphi^{\left(  d \right)  }\right)  \left(  v\right)
:=\sum_{k=0}^{l}\varphi\left(  s_{k}\right)  L_{s_{k}}^{\left(  d \right)
}\left(  v\right)  .
\]

\noindent The domain of extrapolation is denoted $U := D\diagdown S$. It is
convenient to define a generic point $v \in D$ stating that it is an
\emph{observed point} if it is a knot, an \emph{interpolation point} if $v \in
S$ and an \emph{extrapolation point} if $v \in U$.

\noindent For all $d \in\mathbb{N}$, for any $v \in S$, the Lagrange
interpolation scheme converges for the function $\varphi^{(d)}$, that is, for
$l \to\infty$,
\[
\mathcal{L}_{l}\left(  \varphi^{\left(  d \right)  }\right)  \left(  s\right)
\rightarrow\varphi^{\left(  d \right)  }\left(  s\right)  , \quad\forall s\in
S.
\]

\noindent Interpolating the derivative $\varphi^{\left(  d+i\right)  }\left(
s^{\ast}\right)  $ at a point $s^{\ast}\in S$ opportunely chosen, a Taylor
expansion with order $(m-1)$ of $\varphi^{\left(  d \right)  } (v)$ at point
$v$ from $s^{\ast}$ gives
\[
T_{\varphi^{\left(  d \right)  },m,l}\left(  v\right)  :=\sum_{i=0}^{m-1}%
\frac{\left(  v-s^{\ast}\right)  }{i!}^{i}\mathcal{L}_{l}\left(
\varphi^{\left(  d +i\right)  }\right)  \left(  s^{\ast}\right)  , \quad
s^{\ast}\in S,
\]
\noindent and we have
\[
\lim_{m\rightarrow\infty}\lim_{l\rightarrow\infty}T_{\varphi^{\left(  d
\right)  },m,l}\left(  v\right)  =\varphi^{\left(  d \right)  }\left(
v\right)  , \quad\forall v\in D.
\]

\noindent When $\varphi^{(d)} \in\mathcal{C}^{\alpha}(D) $, $\forall\alpha$,
$l\geq2\alpha-3$, the upper bound for the error of approximation is given in
\cite{1},
\[
E_{t}:=\sup_{v \in D}\left\vert \varphi^{\left(  d \right)  }\left(  v\right)
-T_{\varphi^{\left(  d \right)  },m,l}\left(  v\right)  \right\vert \leq
M(m,l, \alpha),
\]

\noindent where $M(m,l, \alpha) =A\left(  \alpha,l\right)  +B\left(  m\right)
$,
\[
A\left(  \alpha,l\right)  :=K\left(  \alpha,l\right)  \sum_{i=0}^{m-1}\left(
\sup_{s\in S}\left\vert \varphi^{\left(  d+i+\alpha\right)  }\left(  s\right)
\right\vert \frac{1}{i!}\sup_{v\in U}\left\vert v-s^{\ast}\right\vert
^{i}\right)  ,
\]
\[
K\left(  \alpha,l\right)  :=\left(  \frac{\pi}{2\left(  1+l\right)  }\left(
\overline{s}-\underline{s}\right)  \right)  ^{\alpha}\left(  9+\frac{4}{\pi
}\ln\left(  1+l\right)  \right)  ,
\]
\[
\text{and } \quad B\left(  m\right)  :=\sup_{v\in\left(  a,\overline
{s}\right)  }\left(  \frac{\left\vert u-s^{\ast}\right\vert ^{m}\left\vert
\varphi^{\left(  d+\alpha\right)  }\left(  v\right)  \right\vert }{m!}\right)
.
\]

\noindent The optimal design writes $\left\{  \left(  n_{k},s_{k }\right)
\in\left(  \mathbb{N}\setminus\{0\} \right)  ^{l+1} \times\mathbb{R} ^{l+1}, ~
n:=\sum_{k=0}^{l}n_{k}, ~ n\text{ fixed}\right\}  $, where $n$ is the total
number of experiments and the $(l+1)$ knots are defined by
\[
s_{k}:=\frac{\overline{s}+\underline{s}}{2}-\frac{\overline{s}-\underline{s}%
}{2}\cos\frac{2k-1}{2l+2}\pi, \qquad k=0, \ldots, l,
\]
with $n_{k}:=\left[  \frac{n\sqrt{{P}_{k}}}{\sum_{k=0}^{l}\sqrt{{P}_{k}}%
}\right]  $, $\left[  .\right]  $ denoting the integer part function, and (see
\cite{3} for details)
\[
P_{k}:=\left\vert \sum_{\beta=0}^{m}\sum_{\alpha=0}^{m}\frac{\left(
u-s\right)  ^{\alpha+\beta}}{\alpha!\beta!}L_{s_{k}}^{\left(  \alpha\right)
}\left(  s\right)  L_{s_{k}}^{\left(  \beta\right)  }\left(  s\right)
\right\vert , \qquad k=0,...,l.
\]

\noindent The function $\varphi$ cannot be observed exactly at the knots.
\noindent Let $\widehat{\varphi\left(  s_{k}\right)  }$ denote the least
squares estimate of $\varphi\left(  s_{k}\right)  $ at the knot $s_{k}$ and
\begin{equation}
\label{lagsch}\mathcal{L}_{l}\left(  \widehat{\varphi^{\left(  d+i\right)  }%
}\right)  \left(  v\right)  :=\sum_{k=0}^{l}\widehat{\varphi\left(
s_{k}\right)  }L_{s_{k}}^{\left(  d+i\right)  }\left(  v\right)  .
\end{equation}
We estimate the $d -$th derivative of $\varphi\left(  v\right)  $ at $v\in D$
as follows
\[
\widehat{T}_{\varphi^{\left(  d \right)  },m,l}\left(  v\right)  :=\sum
_{i=0}^{m-1}\frac{\left(  v-s^{\ast}\right)  }{i!}^{i}\mathcal{L}_{l}\left(
\widehat{\varphi^{\left(  d +i\right)  }}\right)  \left(  s^{\ast}\right)  ,
\qquad s^{\ast}\in S.
\]

\noindent The knots $s_{k}$ are chosen in order to minimize the variance of
$\widehat{T}_{\varphi^{\left(  d \right)  },m,l}\left(  v\right)  $ and it
holds
\[
\lim_{m\rightarrow\infty}\lim_{l\rightarrow\infty}{\lim}_{\min_{k=0, \ldots,
l} \left(  n_{k}\right)  \rightarrow\infty}\widehat{T}_{\varphi^{\left(  d
\right)  },m,l}\left(  v\right)  =\varphi^{\left(  d \right)  }\left(
v\right)  , \qquad\forall v \in D.
\]
$\widehat{T}_{\varphi^{\left(  d \right)  },m,l}\left(  v\right)  $ is an
extrapolation estimator when $v\in U$ and an interpolation estimator when
$v\in S.$

\noindent For a fixed degree $l$ of the Lagrange scheme (\ref{lagsch}), the
total error committed while substituting $\varphi^{\left(  d \right)  }\left(
v\right)  $ by $\widehat{T}_{\varphi^{\left(  d \right)  },m,l}\left(
v\right)  $ writes%

\[
E_{Tot}\left(  \varphi^{\left(  d \right)  }\left(  v\right)  \right)
:=\varphi^{\left(  d \right)  }\left(  v \right)  -\widehat{T}_{\varphi
^{\left(  d \right)  },m,l}\left(  v\right)  .
\]

\noindent For the interpolation error concerning $\varphi^{\left(  i+ d
\right)  }$, we have the following result presented in \cite{6}, p.293 : if
$\varphi^{\left(  i+ d \right)  }\in\mathcal{C}^{ \alpha}\left(  S\right)  $,
$\forall\alpha$, $l\geq2 \alpha-3$, then
\[
\sup_{s \in S}\left\vert \varphi^{\left(  d +i\right)  }\left(  s\right)
-\mathcal{L}_{l}\left(  \varphi^{\left(  d +i\right)  }\right)  \left(
s\right)  \right\vert \leq M_{1}:=K\left(  \alpha,l\right)  \sup_{s\in
S}\left\vert \varphi^{\left(  d +i+ \alpha\right)  }\left(  s\right)
\right\vert .
\]

\noindent This error depends on the very choice of the knots and is controlled
through a tuning of $l$.

\noindent The error due to the Taylor expansion of order $\left(  m-1\right)
$
\[
\varphi^{\left(  d \right)  }\left(  v\right)  -\sum_{i=0}^{m-1}\frac{\left(
v-s^{\ast}\right)  }{i!}^{i}\varphi^{\left(  d +i\right)  }\left(  s^{\ast
}\right)
\]
\noindent depends on $s^{\ast}$, it is a truncation error and it can be
controlled through a tuning of $m$.

\noindent Let $\widehat{\varphi(s_{k})}$ be an estimate of $\varphi(s_{k})$ on
the knot $s_{k}$ and
\[
\varepsilon\left(  k\right)  :=\varphi\left(  s_{k}\right)  -\widehat
{\varphi\left(  s_{k}\right)  }, \qquad k=0,...,l
\]
\noindent denote the error pertaining to $\varphi(s_{k})$ due to this
estimation. $\varepsilon\left(  k\right)  $ clearly depends on $n_{k}$, the
frequency of observations at knot $s_{k}$.

\noindent Finally, when $n$ is fixed, the error committed while extrapolating
depends on the design $\{ (n_{k} , s_{k}) \in\left(  \mathbb{N}\setminus\{ 0
\} \right)  ^{l+1} \times\mathbb{R}^{l+1}, ~ k=0, \ldots, l, ~n = \sum
_{k=0}^{l} n_{k} \}$, on $m$ and on $l$.

\noindent Without loss of generality, we will assume $\sigma=1$. In this case
we have $\widehat{\varphi\left(  s_{k}\right)  }$ =$\overline{Y}\left(
s_{k}\right)  :=\frac{\sum_{j=1}^{n_{k}}Y_{j}\left(  k\right)  }{n_{k}}$. The
general case when $\sigma$ is unknown is described in \cite{3}.

\noindent In the next Section we will provide upper bounds for the errors in
order to control them.

\noindent Since $\varphi$ is supposed to be an analytic function, we can
consider the extrapolation as an analytic continuation of the function out of
the set $S$ obtained by a Taylor expansion from an opportunely chosen point
$s^{\ast}$ in $S.$ So, the extrapolation error will depend on the order of the
Taylor expansion and on the precision in the knowledge of the derivatives of
the function at $s^{\ast}$. This precision is given by the interpolation error
and by the estimation errors on the knots. The analyticity assumption also
implies that the interpolation error will quickly converge to zero. Indeed,
for all integer $r$, the following result holds:%

\[
\lim_{l\rightarrow\infty}l^{r}\sup_{s\in S}\left\vert \varphi^{\left(
j\right)  }\left(  s\right)  -\sum_{k=0}^{l}L_{s_{k}}^{\left(  j\right)
}\left(  s\right)  \varphi\left(  s_{k}\right)  \right\vert =0.
\]

\noindent We remark that the instability of the interpolation and
extrapolation schemes discussed by Runge (1901) can be avoided if the chosen
knots form a Tchebycheff set of points in $S$, or if they form a Feteke set of
points in $S$, or by using splines.

\noindent Note that in all the works previously quoted the function is
supposed to be polynomial with known degree (in \cite{10} and \cite{11}), to
belongs to a Sobolev space (see \cite{17}, \cite{18}, \cite{19} and
\cite{20}), or to be quasi analytic (in \cite{4} and \cite{5}), or analytic
(in \cite{3}). Moreover, $\widetilde{S}$ is chosen as a Tchebycheff set of
points in $S$ .

\noindent Bernstein in \cite{2} affirmed that polynomials of low degree are
good approximations for analytic functions. In the case of the
Broniatowski-Celant design (\cite{3}), the double approximation to approach
$\varphi$ allows to choose any subset of $S$ as possible interpolation set.
So, if the unknown function is supposed to be analytic, then we can choose a
small interpolation set in order to obtain a small interpolation error.

\section{Upper bounds and control of the error}

The extrapolation error depends on three kinds of errors: truncation error,
interpolation error and error of estimation of the function on the knots. In
order to control the extrapolation error, we split an upper bound for it in a
sum of three terms, each term depending only on one of the three kinds of errors.

\noindent In the sequel, we will distinguish two cases: in the first case, we
suppose that the observed random variable $Y$ is bounded, in the second case
$Y$ is supposed to be a random variable with unbounded support. We suppose
that the support is known.

\subsection{\bigskip Case 1: $Y$ is a bounded random variable}

If $\tau_{1},\tau_{2}$ (assumed known) are such that $\Pr\left(  \tau_{1}\leq
Y\leq\tau_{2}\right)  =1$, it holds $\left\vert \varphi\left(  v\right)
\right\vert \leq R$, where $R:=\max\left\{  \left\vert \tau_{1}\right\vert
,\left\vert \tau_{2}\right\vert \right\}  .$ Indeed, $E\left(  Y\right)
=\varphi\in\left[  -R,R\right]  .$ Let%

\[
\varepsilon\left(  k\right)  :=\frac{\sum_{j=1}^{n_{k}}Y_{j}\left(  k\right)
}{n_{k}}-\varphi\left(  s_{k}\right)  .
\]

\noindent The variables $Y_{j}\left(  k\right)  ,\forall j=1,...,n_{k},\forall
k=0,..,l,$ are i.i.d., with the same bounded support and for all $k$$,
E\left(  Y_{j}\left(  k\right)  \right)  =$ $\varphi\left(  s_{k}\right)  $,
hence we can apply the Hoeffding's inequality (in \cite{9}):%

\[
\Pr\left\{  \left\vert \varepsilon\left(  k\right)  \right\vert \geq
\rho\right\}  \leq2\exp\left(  -\frac{2\rho^{2}n_{k}}{\left(  \tau_{2}%
-\tau_{1}\right)  ^{2}}\right)  .
\]

\noindent In Proposition 1, we give an upper bound for the extrapolation error
denoted by $E_{ext}$. This bound is the sum of the three terms, $M_{Taylor}$,
controlling the error associated to the truncation of the Taylor expansion
which defines $\varphi^{\left(  d\right)  }$, $M_{interp}$, controlling the
interpolation error and $M_{est}$, describing the estimation error on the knots.

\begin{proposition}
For all $\alpha\in\mathbb{N} \setminus\{ 0 \} $, if $\varphi^{\left(
i+d\right)  }\in\mathcal{C}^{\alpha} \left(  a,b\right)  $, $l\geq2\alpha-3$,
then, $\forall u \in U$, $\left\vert E_{ext}\left(  u\right)  \right\vert \leq
M_{Taylor}+M_{interp}+M_{est}$, where
\[
M_{Taylor}:=R\frac{\left(  d+m\right)  !}{m!}\left(  \frac{s^{\ast}-u}%
{b-a}\right)  ^{m}\frac{1}{\left(  b-a\right)  ^{d}},
\]
\[
K\left(  l,\alpha\right)  :=\left(  9+\frac{4}{\pi}\ln\left(  1+l\right)
\right)  \left(  \frac{\pi}{2\left(  1+l\right)  }\right)  ^{\alpha},
\]
\[
M_{interp}:=K\left(  l,\alpha\right)  \frac{R}{\left(  \overline{s}%
-\underline{s}\right)  ^{d+\alpha}}\sum_{i=0}^{m-1}\left(  \frac{s^{\ast}%
-u}{\overline{s}-\underline{s}}\right)  ^{i}\frac{\left(  d+i+\alpha\right)
!}{i!},
\]
\[
\Lambda\left(  l,m\right)  :=\sum_{i=0}^{m-1}\sum_{k=0}^{l}\frac{\left(
s^{\ast}-u\right)  ^{i}}{i!}\left\vert L_{s_{k}}^{\left(  d+i\right)  }\left(
s^{\ast}\right)  \right\vert ,
\]
\[
M_{est}:=\Lambda\left(  l,m\right)  \left(  \max_{k=0,...,l}\left\vert
\varepsilon\left(  k\right)  \right\vert \right)  .
\]

\end{proposition}

\begin{proof}
By using the Cauchy's Theorem on the derivatives of the analytic functions, we
obtain
\[
\left\vert \varphi^{\left(  d\right)  }\left(  u\right)  -\widehat
{\varphi^{\left(  d\right)  }\left(  u\right)  }\right\vert =\left\vert
\varphi^{\left(  d\right)  }\left(  u\right)  +\sum_{i=0}^{m-1}\frac
{\varphi^{\left(  d+i\right)  }\left(  s^{\ast}\right)  }{i!}\left(
u-s^{\ast}\right)  ^{i}-\sum_{i=0}^{m-1}\frac{\varphi^{\left(  d+i\right)
}\left(  s^{\ast}\right)  }{i!}\left(  u-s^{\ast}\right)  ^{i}-\widehat
{\varphi^{\left(  d\right)  }\left(  u\right)  }\right\vert
\]%
\[
\leq\left\vert \varphi^{\left(  d\right)  }\left(  u\right)  -\sum_{i=0}%
^{m-1}\frac{\varphi^{\left(  d+i\right)  }\left(  s^{\ast}\right)  }%
{i!}\left(  u-s^{\ast}\right)  ^{i}\right\vert +\left\vert \sum_{i=0}%
^{m-1}\frac{\varphi^{\left(  d+i\right)  }\left(  s^{\ast}\right)  }%
{i!}\left(  u-s^{\ast}\right)  ^{i}-\widehat{\varphi^{\left(  d\right)
}\left(  u\right)  }\right\vert
\]%
\[
\leq\frac{\sup_{v\in U}\left\vert \varphi^{\left(  d+m\right)  }\left(
v\right)  \right\vert }{m!}\left(  s^{\ast}-u\right)  ^{m}+\left\vert
\sum_{i=0}^{m-1}\frac{\varphi^{\left(  d+i\right)  }\left(  s^{\ast}\right)
}{i!}\left(  u-s^{\ast}\right)  ^{i}-\sum_{i=0}^{m-1}\frac{\widehat
{\varphi^{\left(  d+i\right)  }\left(  s^{\ast}\right)  }}{i!}\left(
u-s^{\ast}\right)  ^{i}\right\vert
\]%
\[
\leq\frac{R\left(  m+d\right)  !}{\left(  b-a\right)  ^{d}m!}\left(
\frac{s^{\ast}-u}{b-a}\right)  ^{m}+\left\vert \sum_{i=0}^{m-1}\frac{\left(
s^{\ast}-u\right)  ^{i}}{i!}\left(  \varphi^{\left(  d+i\right)  }\left(
s^{\ast}\right)  -\widehat{\varphi^{\left(  d+i\right)  }\left(  s^{\ast
}\right)  }\right)  \right\vert
\]%
\[
\leq\frac{R\left(  m+d\right)  !}{\left(  b-a\right)  ^{d}m!}\left(
\frac{s^{\ast}-u}{b-a}\right)  ^{m}+\sum_{i=0}^{m-1}\frac{\left(  s^{\ast
}-u\right)  ^{i}}{i!}\left\vert \varphi^{\left(  d+i\right)  }\left(  s^{\ast
}\right)  -\widehat{\varphi^{\left(  d+i\right)  }\left(  s^{\ast}\right)
}\right\vert
\]%
\[
\leq M_{Taylor}+\sum_{i=0}^{m-1}\frac{\left(  s^{\ast}-u\right)  ^{i}}%
{i!}\left\vert
\begin{array}
[c]{c}%
\varphi^{\left(  d+i\right)  }\left(  s^{\ast}\right)  -\sum_{k=0}^{l}%
L_{s_{k}}^{\left(  d+i\right)  }\left(  s^{\ast}\right)  \varphi\left(
s_{k}\right) \\
+\sum_{k=0}^{l}L_{s_{k}}^{\left(  d+i\right)  }\left(  s^{\ast}\right)
\varphi\left(  s_{k}\right)  -\widehat{\varphi^{\left(  d+i\right)  }\left(
s^{\ast}\right)  }%
\end{array}
\right\vert
\]%
\begin{align*}
&  \leq M_{Taylor}+\sum_{i=0}^{m-1}\frac{\left(  s^{\ast}-u\right)  ^{i}}%
{i!}\left\vert \varphi^{\left(  d+i\right)  }\left(  s^{\ast}\right)
-\sum_{k=0}^{l}L_{s_{k}}^{\left(  d+i\right)  }\left(  s^{\ast}\right)
\varphi\left(  s_{k}\right)  \right\vert +\\
&  \sum_{i=0}^{m-1}\sum_{k=0}^{l}\frac{\left(  s^{\ast}-u\right)  ^{i}}%
{i!}L_{s_{k}}^{\left(  d+i\right)  }\left(  s^{\ast}\right)  \left\vert
\varphi\left(  s_{k}\right)  -\overline{Y}\left(  k\right)  \right\vert
\end{align*}%
\begin{align*}
&  \leq M_{Taylor}+\sum_{i=0}^{m-1}\frac{\left(  s^{\ast}-u\right)  ^{i}}%
{i!}K\left(  l,\alpha\right)  \left(  \sup_{s\in S}\left\vert \varphi^{\left(
d+i+\alpha\right)  }\left(  s\right)  \right\vert \right) \\
&  +\sum_{i=0}^{m-1}\sum_{k=0}^{l}\frac{\left(  s^{\ast}-u\right)  ^{i}}%
{i!}L_{s_{k}}^{\left(  d+i\right)  }\left(  s^{\ast}\right)  \left\vert
\varphi\left(  s_{k}\right)  -\overline{Y}\left(  k\right)  \right\vert
\end{align*}%
\begin{align*}
&  \leq M_{Taylor}+\frac{R}{\left(  \overline{s}-\underline{s}\right)
^{d+\alpha}}\sum_{i=0}^{m-1}\frac{\left(  s^{\ast}-u\right)  ^{i}}{i!}K\left(
l,\alpha\right)  \frac{\left(  d+i+\alpha\right)  !}{\left(  \overline
{s}-\underline{s}\right)  ^{i}}\\
&  +\sum_{i=0}^{m-1}\sum_{k=0}^{l}\frac{\left(  s^{\ast}-u\right)  ^{i}}%
{i!}L_{s_{k}}^{\left(  d+i\right)  }\left(  s^{\ast}\right)  \left\vert
\varphi\left(  s_{k}\right)  -\overline{Y}\left(  k\right)  \right\vert
\end{align*}%
\[
\leq M_{Taylor}+M_{interp}+\sum_{i=0}^{m-1}\sum_{k=0}^{l}\frac{\left(
s^{\ast}-u\right)  ^{i}}{i!}L_{s_{k}}^{\left(  d+i\right)  }\left(  s^{\ast
}\right)  \left\vert \varphi\left(  s_{k}\right)  -\overline{Y}\left(
k\right)  \right\vert
\]%
\begin{align*}
&  \leq M_{Taylor}+M_{interp}+\left(  \max_{k=0,...,l}\left\vert
\varepsilon\left(  k\right)  \right\vert \right)  \sum_{i=0}^{m-1}\sum
_{k=0}^{l}\frac{\left(  s^{\ast}-u\right)  ^{i}}{i!}\left\vert L_{s_{k}%
}^{\left(  d+i\right)  }\left(  s^{\ast}\right)  \right\vert \\
&  =M_{Taylor}+M_{interp}+M_{est}.
\end{align*}

\end{proof}

\noindent Proposition 2 yields the smallest integer such that the error of
estimation is not greater than a chosen threshold with a fixed probability.

\begin{proposition}
$\forall\eta\in\left[  0,1\right]  ,\forall\rho\in\mathbb{R} ^{+},\exists
n\in\mathbb{N} $ such that
\[
Pr \left(  \max_{k=0,...,l}\left\vert \varepsilon\left(  k\right)  \right\vert
\geq\frac{\rho}{\Lambda\left(  l,m\right)  }\right)  \leq\eta.
\]

\end{proposition}

\begin{proof}
If, $\forall k$ $\left\vert \varepsilon\left(  k\right)  \right\vert \geq
\frac{\rho}{\Lambda\left(  l,m\right)  }$, then $\max_{k=0,...,l}\left\vert
\varepsilon\left(  k\right)  \right\vert \geq\frac{\rho}{\Lambda\left(
l,m\right)  }.$ We have
\[
\ \Pr\left(  \max_{k=0,...,l}\left\vert \varepsilon\left(  k\right)
\right\vert \geq\frac{\rho}{\Lambda\left(  l,m\right)  }\right)  \leq
\prod_{k=0}^{l}\Pr\left(  \left\vert \varepsilon\left(  k\right)  \right\vert
\geq\frac{\rho}{\Lambda\left(  l,m\right)  }\right)  \leq\prod_{k=0}^{l}%
2\exp\left(  -\frac{2\rho^{2}}{\left(  \Lambda\left(  l,m\right)  \right)
^{2}}n_{k}\right)  .
\]

\noindent So, we can choose
\[
n^{\ast}=\left[  \frac{\left(  l+1\right)  \ln2-\ln\eta}{2}\left(
\frac{\Lambda\left(  l,m\right)  \left(  \tau_{2}-\tau_{1}\right)  }{\rho
}\right)  ^{2}\right]  .
\]

\end{proof}

\noindent Proposition 3 gives an upper bound for the extrapolation error that
depends on $\left(  l,m,n\right)  $. We recall that the number of knots $l+1$
controls the interpolation error, $m$ denotes the number of terms used in the
Taylor expansion for $\varphi^{\left(  d\right)  }$ and $n$ is the total
number of observations used to estimate $\varphi\left(  s_{k}\right)
,k=0,..,l $. Hence $n$ controls the total estimation error.

\begin{proposition}
With the same hypotheses and notations, we have that
\[
\forall\left(  \rho_{m},\rho_{l},\rho_{n}\right)  \in%
\mathbb{R}
( \mathbb{R} ^{+})^{3}, \quad\left\vert E_{ext}\left(  u\right)  \right\vert
\leq\rho_{m}+\rho_{l}+\rho_{n}%
\]
\noindent with probability $\eta$. $\eta$ depends on the choice of $\left(
\rho_{m},\rho_{l},\rho_{n}\right)  $, which depends on $\left(  m,l,n\right)
.$
\end{proposition}

\begin{proof}
When $\left(  \rho_{m},\rho_{l}\right)  $ is fixed , we can choose $\left(
m,l\right)  $ as the solution of the system:
\[
\left(  M_{Taylor},M_{interp}\right)  =\left(  \rho_{m},\rho_{l}\right)  .
\]
\noindent We end the proof by taking $\rho_{n}=\frac{\rho}{\Lambda\left(
l,m\right)  }$ and $n=n^{\ast}$. \newline
\end{proof}

\noindent In the case of the estimation of $\varphi\left(  u\right)  $ (i.e.,
when $d=0$) we obtain for the couple $\left(  m,n\right)  $ the explicit
solution
\[
m=\frac{\ln\rho_{m}-\ln R}{\ln\left(  s^{\ast}-u\right)  -\ln\left(
b-a\right)  },
\]
\[
n=\left[  \frac{\left(  l+1\right)  \ln2-\ln\eta}{2}\left(  \frac
{\Lambda\left(  l\right)  \left(  \tau_{2}-\tau_{1}\right)  }{\rho}\right)
^{2}\right]  ,\Lambda\left(  l\right)  =\sum_{i=0}^{m-1}\sum_{k=0}^{l}%
\frac{\left(  s^{\ast}-u\right)  ^{i}}{i!}\left\vert L_{s_{k}}^{\left(
i\right)  }\left(  s^{\ast}\right)  \right\vert .
\]

\noindent When $l\geq2\alpha-3$, $l$ is the solution of the equation
\[
\text{ }\rho_{l}=\left(  9+\frac{4}{\pi}\ln\left(  1+l\right)  \right)
\left(  \frac{\pi}{2\left(  1+l\right)  }\right)  ^{\alpha}\frac{R}{\left(
\overline{s}-\underline{s}\right)  ^{\alpha}}\sum_{i=0}^{m-1}\left(
\frac{s^{\ast}-u}{\overline{s}-\underline{s}}\right)  ^{i}\frac{\left(
i+\alpha\right)  !}{i!}.
\]

\noindent Theorem 4, due to Markoff, provides an uniform bound for the
derivatives of a Lagrange polynomial.

\begin{theorem}
(Markoff) Let $P_{l}\left(  s\right)  :=\sum_{j}a_{j}s^{j}$ be a polynomial
with real coefficients and degree $l$. If $\sup_{s\in S}\left\vert
P_{l}\left(  s\right)  \right\vert \leq W,$ then for all $s$ in $intS$ and for
all $l$ in $\mathbb{N} ,$ it holds%

\[
\left\vert P_{l}^{\left(  j\right)  }\left(  s\right)  \right\vert \leq
\frac{l^{2}\left(  l^{2}-1\right)  ...\left(  l^{2}-\left(  j-1\right)
^{2}\right)  }{\left(  2j-1\right)  !!}\left(  \frac{2}{\left(  \overline
{s}-\underline{s}\right)  }\right)  ^{j}W.
\]

\end{theorem}

\noindent When applied to the elementary Lagrange polynomial, it is readily
checked that $W=\pi$. Indeed,%

\[
\left\vert L_{s_{k}}\left(  s\right)  \right\vert =\left\vert \frac{\left(
-1\right)  ^{k}\sin\left(  \frac{2k-1}{2l+2}\pi\right)  }{l+1}\frac
{\cos\left(  \left(  l+1\right)  \theta\right)  }{\cos\theta-\cos\left(
\frac{2k-1}{2l+2}\pi\right)  }\right\vert \leq
\]

\[
\leq\frac{\left\vert \sin\left(  \frac{2k-1}{2l+2}\pi\right)  \right\vert
}{l+1}\frac{\left\vert \cos\left(  \left(  l+1\right)  \theta\right)
\right\vert }{\left\vert \cos\theta-\cos\left(  \frac{2k-1}{2l+2}\pi\right)
\right\vert }\leq
\]

\[
\leq\frac{\left\vert \sin\left(  \frac{2k-1}{2l+2}\pi\right)  \right\vert
}{l+1}\frac{\left(  l+1\right)  \left\vert \theta-\frac{2k-1}{2l+2}%
\pi\right\vert }{\frac{1}{\pi}\sin\left(  \frac{2k-1}{2l+2}\pi\right)
\left\vert \theta-\frac{2k-1}{2l+2}\pi\right\vert }=\pi.
\]

\noindent We used
\[
\left\vert \cos\left(  \left(  l+1\right)  \theta\right)  \right\vert
=\left\vert \cos\left(  \left(  l+1\right)  \theta\right)  -\cos\left(
\left(  l+1\right)  \frac{2k-1}{2l+2}\pi\right)  \right\vert \leq\left(
l+1\right)  \left\vert \theta-\frac{2k-1}{2l+2}\pi\right\vert
\]
\noindent and $\cos\left(  \left(  l+1\right)  \frac{2k-1}{2l+2}\pi\right)
=0$. Moreover,
\[
\left\vert \cos\theta-\cos\left(  \frac{2k-1}{2l+2}\pi\right)  \right\vert
=2\sin\left(  \frac{\theta+\frac{2k-1}{2l+2}\pi}{2}\right)  \left\vert
\sin\left(  \frac{\theta-\frac{2k-1}{2l+2}\pi}{2}\right)  \right\vert .
\]
\noindent The concavity of the sine function on $\left[  0,\pi\right]  $ implies%

\[
\sin\left(  \frac{\theta+\frac{2k-1}{2l+2}\pi}{2}\right)  \geq\frac{1}%
{2}\left(  \sin\theta+\sin\left(  \frac{2k-1}{2l+2}\pi\right)  \right)
\]

\[
\left\vert \sin\left(  \frac{\theta-\frac{2k-1}{2l+2}\pi}{2}\right)
\right\vert \geq\frac{2}{\pi}\left\vert \theta-\frac{2k-1}{2l+2}\pi\right\vert
,\theta\in\left[  0,\pi\right]  .
\]

\begin{remark}
The Cauchy theorem merely gives a rough upper bound. In order to obtain a
sharper upper bound, we would assume some additional hypotheses on the
derivatives of the function.
\end{remark}

\subsection{Case 2: $Y$ is an unbounded random variable}

If the support of the random variable $Y$ is not bounded and $\varphi$ is a
polynomial of unknown degree $t$, $t \leq g-1$, with $g$ known, it's still
possible to give an upper bound for the estimation error. Since%

\[
\varphi^{\left(  d\right)  }=\sum_{i=0}^{g-1}\frac{\varphi^{\left(
d+i\right)  }\left(  s^{\ast}\right)  }{i!}\left(  u-s^{\ast}\right)  ^{i}=
\]

\[
=\sum_{i=0}^{g-1}\frac{\sum_{k=0}^{g-1}L_{s_{k}}^{\left(  d+i\right)  }\left(
s^{\ast}\right)  \varphi\left(  s_{k}\right)  }{i!}\left(  u-s^{\ast}\right)
^{i}=\sum_{k=0}^{g-1}L_{s_{k}}^{\left(  d\right)  }\left(  u\right)
\varphi\left(  s_{k}\right)  ,
\]

\noindent$\varphi^{\left(  d\right)  }$ can be estimated as follows%

\[
\widehat{\varphi^{\left(  d\right)  }\left(  u\right)  }=\sum_{k=0}%
^{g-1}L_{s_{k}}^{\left(  d\right)  }\left(  u\right)  \overline{Y}\left(
s_{k}\right)  .
\]

\noindent We have in probability $\widehat{\varphi^{\left(  d\right)  }%
}\rightarrow\varphi^{\left(  d\right)  }$ for $\min\left(  n_{k}\right)
\rightarrow\infty$. So,
\[
Var\left(  \widehat{\varphi^{\left(  d\right)  }}\right)  =\sum_{k=0}%
^{g-1}\left(  L_{s_{k}}^{\left(  d\right)  }\left(  u\right)  \right)
^{2}\frac{\varsigma}{n_{k}}\rightarrow0,
\]
\noindent where $\varsigma$ is the variance of $Z$. We use the Tchebycheff's
inequality in order to obtain an upper bound for the estimation error. For a
given $\eta$,
\[
\Pr\left\{  \left\vert \widehat{\varphi^{\left(  d\right)  }}-\varphi^{\left(
d\right)  }\right\vert \geq\eta\right\}  \leq\frac{\sum_{k=0}^{g-1}\left(
L_{s_{k}}^{\left(  d\right)  }\left(  u\right)  \right)  ^{2}\frac{\varsigma
}{n_{k}}}{\eta^{2}}.
\]

\noindent If we aim to obtain, for all fixed $\omega$, $\Pr\left\{  \left\vert
\widehat{\varphi^{\left(  d\right)  }}-\varphi^{\left(  d\right)  }\right\vert
\geq\eta\right\}  \leq\omega\text{,} $ we can choose $n^{\ast}$ as the
solution of the equation $\frac{\sum_{k=0}^{g-1}\left(  L_{s_{k}}^{\left(
d\right)  }\left(  u\right)  \right)  ^{2}\frac{\varsigma}{n_{k}}}{\eta^{2}%
}=\omega$, that is
\[
n^{\ast}=\frac{\sum_{k=0}^{g-1}\left(  L_{s_{k}}^{\left(  d\right)  }\left(
u\right)  \right)  ^{2}\varsigma}{\omega\eta^{2}}.
\]
\noindent The integer $\left[  n^{\ast}\right]  $ is such that the inequality
$\Pr\left\{  \left\vert \widehat{\varphi^{\left(  d\right)  }}-\varphi
^{\left(  d\right)  }\right\vert \geq\eta\right\}  \leq\omega$ is satisfied.

\noindent We remark that if we know the degree $t$ of the polynomial, then it
is sufficient to set $g-1=t$. When $\varphi\left(  u\right)  =\varphi
^{d}\left(  u\right)  $ (i.e., $d=0$), we have $\left[  n^{\ast}\right]
=\frac{\sum_{k=0}^{g-1}\left(  L_{s_{k}}\left(  u\right)  \right)
^{2}\varsigma}{\omega\eta^{2}}.$

\noindent We underline that for $d=0$ and when $t$ is known $\widehat
{\varphi^{\left(  d\right)  }\left(  u\right)  }=\widehat{\varphi\left(
u\right)  }$ coincides with Hoel's estimator.

\noindent If the solely information on $\varphi$ is that $\varphi$ is analytic
then we are constrained to give hypotheses on the derivatives of the function.
More precisely, since $\operatorname{Im}\varphi\subseteq\mathbb{R}$, we can't
apply the Cauchy theorem on the analytic functions; we can only say that
$\varphi\left(  v\right)  =E\left(  Y\right)  \in\mathbb{R}$. So, we are not
able to find a constant $R$ such that $\left\vert \varphi\left(  v\right)
\right\vert \leq R$. Moreover, since we can't observe $\varphi\left(
v\right)  $ for $v\notin S$, we don't have any data to estimate $M_{Taylor}$.


\bigskip

\section{Bibliography}

\begin{thebibliography}{99}                                                                                               %


\bibitem[1]{1}Bennett, G., 1962. Probability Inequalities for the Sum of
Independent Random Variables, \emph{Journal of the American Statistical
Association}, 57, 297,33--45.



\bibitem[2]{2}Bernstein, S.N., 1918. Quelques remarques sur l'interpolation,
\emph{Math. Ann.}, 79, 1--12.

\bibitem[3]{3}Broniatowski, M., and Celant, G., 2007. Optimality and bias of
some interpolation and extrapolation designs. \emph{J. Statist. Plann.
Inference}, 137, 858--868.

\bibitem[4]{4}Celant, G., 2003. Extrapolation and optimal designs for
accelerated runs. \emph{Ann. I.S.U.P.}, 47, 3, 51--84.

\bibitem[5]{5}Celant, G., 2002. Plans acc\'{e}l\'{e}r\'{e}s optimaux:
estimation de la vie moyenne d'un syst\`{e}me. \emph{C. R. Math. Acad. Sci.
Paris}, 335, 1, 69--72.

\bibitem[6]{6}Coatm\'{e}lec, C., 1966. Approximation et interpolation des
fonctions diff\'{e}rentiables de plusieurs variables. \emph{Ann. Sci. Ecole
Norm. Sup.}, 83, 4, 271--341.







\bibitem[7]{9}Hoeffding, W., 1963. Probability inequalities for sums of
bounded random variables. \emph{J. Amer. Statist. Assoc.}, 58, 13--30.

\bibitem[8]{10}Hoel, P.G., Levine, A., 1964. Optimal spacing and weighting in
polynomial prediction. \emph{Ann. Math. Statist.}, 35, 1553--1560.

\bibitem[9]{11}Hoel, P.G., 1965. Optimum designs for polynomial extrapolation.
\emph{Ann. Math. Statist.}, 36, 1483--1493.





\bibitem[10]{14}Rivlin, T.J., 1969. An introduction to the approximation of
functions. Blaisdell Publishing Co. Ginn and Co., Waltham, Mass.-Toronto, Ont.-London





\bibitem[11]{16}Sch\"{o}nhage, A., 1961. Fehlerfortplanzung bei interpolation,
\emph{Numer. Math.}, 3, 62--71.

\bibitem[12]{17}Spruill, M.C., 1984. Optimal designs for minimax
extrapolation. \emph{J. Multivariate Anal.}, 15, 1, 52--62.

\bibitem[13]{18}Spruill, M.C., 1987. Optimal designs for interpolation.
\emph{J. Statist. Plann. Inference}, 16, 2, 219--229.

\bibitem[14]{19}Spruill, M.C., 1987. Optimal extrapolation of derivatives.
\emph{Metrika}, 34, 1, 45--60.

\bibitem[15]{20}Spruill, M.C., 1990. Optimal designs for multivariate
interpolation. \emph{J. Multivariate Anal.}, 34, 1, 141--155.


\end{thebibliography}
\end{document}